\documentclass[12pt]{amsart}
\usepackage{amsmath,amssymb,amsfonts,amsthm,amsopn}
\usepackage{graphics}
\usepackage{xcolor}
\usepackage{color, colortbl}
\usepackage{pb-diagram}
\usepackage[title]{appendix}

 \def\Spnr{Sp(d,\R)}
 \def\Gltwonr{GL(2d,\R)}

\newcommand{\tfs}{time-frequency shift}

\newtheorem{theorem}{Theorem}[section]

\newtheorem{corollary}[theorem]{Corollary}
\newtheorem{proposition}[theorem]{Proposition}
\newtheorem{definition}[theorem]{Definition}

\newtheorem{example}[theorem]{Example}

\newcommand{\beqa}{\begin{eqnarray*}}
\newcommand{\eeqa}{\end{eqnarray*}}

\newcommand{\field}[1]{\mathbb{#1}}
\newcommand{\bR}{\field{R}}        
\newcommand{\bN}{\field{N}}        
\newcommand{\bZ}{\field{Z}}        
\newcommand{\bC}{\field{C}}        
\def\la{\lambda}

\def\cF{\mathcal{F}}              
\def\cS{\mathcal{S}}
\def\cD{\mathcal{D}}

\def\cA{\mathcal{A}}

\def\cC{\mathcal{C}}

\def\a{\aleph}

\def\rd{\bR^d}

\def\rdd{{\bR^{2d}}}

\def\lrd{L^2(\rd)}
\def\lrdd{L^2(\rdd)}

\def\intrd{\int_{\rd}}
\def\intrdd{\int_{\rdd}}

\def\R{\right)}

\def\<{\left<}
\def\>{\right>}

\def\inv{^{-1}}

\def\mv1{M_v^1}

\def\phas{(x,\o )}
\def\mn{(m,n)}
\def\mn'{(m',n')}

\def\Spnr{Sp(d,\R)}

\def\o{\xi}
\def\a{\alpha}
\def\b{\beta}

\def\N{\mathbb{N}}
\def\R{\mathbb{R}}
\def\Ren{\mathbb{R}^d}

\def\sch{\mathcal{S}}

\def\Tau{\mathcal{T}}

\def\Sn2{S_{2}(L^{2}(\Ren))}
\def\S1{S_{1}(L^{2}(\Ren))}
\def\sig00{\sigma_{0,0}}

\def\la{\langle}
\def\ra{\rangle}

\begin{document}
\begin{abstract} We study the decay properties of Wigner kernels for Fourier integral operators of types I and II. The symbol spaces that allow a nice decay of these kernels are the Shubin classes $\Gamma^m(\rdd)$, with negative order $m$. The phases considered are the so-called tame ones, which appear in the Schr\"odinger propagators. The related canonical transformations are allowed to be nonlinear. It is the nonlinearity of these transformations that are the main obstacles for  nice kernel localizations when symbols are taken in the  H\"ormander's class $S^{0}_{0,0}(\rdd)$. Here we prove that Shubin classes  overcome this problem and allow a nice kernel localization, which improves with the decreasing of the order $m$.
\end{abstract}

\title[Wigner Analysis of FIOS]{Wigner Analysis of Fourier Integral Operators with symbols in the Shubin classes}

\author{Elena Cordero}
\address{Universit\`a di Torino, Dipartimento di Matematica, via Carlo Alberto 10, 10123 Torino, Italy}
\email{elena.cordero@unito.it}
\author{Gianluca Giacchi}
\address{Universit\`a di Bologna, Dipartimento di Matematica, Piazza di Porta San Donato 5, 40126 Bologna, Italy; University of Lausanne, Switzerland; HES-SO School of Engineering, Rue De L'Industrie 21, Sion, Switzerland; Centre Hospitalier Universitaire Vaudois, Switzerland}
\email{gianluca.giacchi2@unibo.it}
\author{Luigi Rodino}
\address{Universit\`a di Torino, Dipartimento di Matematica, via Carlo Alberto 10, 10123 Torino, Italy}
\email{luigi.rodino@unito.it}
\author{Mario Valenzano}
\address{Universit\`a di Torino, Dipartimento di Matematica, via Carlo Alberto 10, 10123 Torino, Italy}
\email{mario.valenzano@unito.it}
\thanks{}

\maketitle

\section{Introduction}

The protagonist of this study is the Wigner distribution, one of the most popular time-frequency representations. It was introduced by  Wigner in  1932 \cite{Wigner} in the framework of Quantum Mechanics and  later applied to signal processing and time-frequency analysis  by Ville, Cohen and  many other authors, see, e.g., \cite{Cohen1,Cohen2,Ville48} and the textbooks \cite{Elena-book,Gosson-Wigner,book}.
\begin{definition} Consider $f,g\in\lrd$. 
	The cross-Wigner distribution $W(f,g)$ is
	\begin{equation}\label{CWD}
		W(f,g)(x,\xi)=\intrd f(x+\frac t2)\overline{g(x-\frac t2)}e^{-2\pi i t\xi}\,dt,\quad \phas\in\rdd.
	\end{equation} If $f=g$ we write $Wf:=W(f,f)$, the so-called Wigner distribution of $f$.
\end{definition}
 Wigner used the above representation to analyse the action of the  Schr\"odinger propagators.
We may extend  the Wigner approach in \cite{Wigner} as follows:  given a linear operator $T:\,\cS(\rd)\to \cS'(\rd)$, we consider  an operator $K$ on $\cS(\rdd)$ such that 
\begin{equation}\label{I3}
	W(Tf,Tg) = KW(f,g), \qquad f,g\in\cS(\rd).
\end{equation}
Its  integral kernel $k$ is called the \emph{Wigner kernel}  of $T$:
\begin{equation}\label{I4}
	W(Tf,Tg)(z) = \intrdd k(z,w) W(f,g)(w)\,dw,\quad z\in\rdd,\quad f,g\in\cS(\rd).
\end{equation}
As an elementary example of the effectiveness of the Wigner distribution, consider the Schr\"odinger propagator $T_\tau$, for a fixed time $\tau\in\bR$, of the free particle equation
$$T_\tau f(x)=\intrd e^{2\pi i( x\xi-\tau
	\xi^2)} {\widehat {f}}(\xi)d\xi,\quad x\in\rd.
$$
We have 
$$W(T_\tau f) (x,\xi)=Wf(x-\tau\xi,\xi)$$
with Wigner kernel 
\begin{equation}\label{L1}
	k=\delta_{z-\chi(w)},\quad w,z\in\rdd,
\end{equation}
where, if we write $w=(y,\eta)$, then $\chi(y,\eta)=(y+\tau\eta,\eta)$.  So we recapture exactly the inertial first Newton's law, from a probabilistic point of view. This striking result is due to the peculiar action of $W$ on the phase $\Phi(x,\xi)=x\xi-\tau\xi^2$. It generalizes to quadratic $\Phi\phas$, corresponding to quadratic Hamiltonians and linear symplectic map $\chi$ in \eqref{L1}, see for example \cite{Gos11}. 

Our aim is to extend this analysis to more general operators, namely Fourier integral operators of the form 
\begin{equation}\label{L2}
T_If(x)=\int_{\rd}e^{2\pi i\Phi(x,\xi)}\sigma(x,\xi)\hat f(\xi)d\xi, \qquad f\in\cS(\rd),
\end{equation}
with phase $\Phi$ and symbol $\sigma$ in suitable classes.
A preliminary step was presented in \cite{CRPartI2022}, with $T$ a pseudodifferential operator $\sigma(x,D)$, i.e., $\Phi\phas=x\xi$ in \eqref{L2}.  The case of a quadratic $\Phi$ and a general $\sigma$ was considered in \cite{CGRPartII2022} and in \cite{CRGFIO1}, where a generalization of \eqref{L1} was obtained by combining a linear symplectic map $\chi$ with the kernel of a pseudodifferential operator. 

In the present paper we focus on the case of nonlinear symplectic mappings $\chi$ corresponding to non-quadratic $\Phi$, which we call \emph{tame}, see Section $2$ below for their definition.

As a counterpart of \eqref{L1} we look for estimates of the type 
\begin{equation}\label{L3}
	|k(z,w) |\lesssim \frac{1}{\la z-\chi(w)\ra^{2N}},
\end{equation}
in the spirit of the estimates for Gabor kernels, which have been widely investigated in the literature, classical references are \cite{BC2021,sparsity2009,CGNRJMPA,CGNRJMP2014,charly06,GR}, see also \cite[Chapter 5]{Elena-book}.

There are two obstructions to the validity of \eqref{L3}. The first, evident from \eqref{L1} and also in the linear case, is that $k(z,w)$ is not point-wise defined for $z=\chi(w)$. This can be easily rephrased by a rescaling of regularity. The second obstruction is of deeper nature,  and it concerns only the nonlinear symplectic map $\chi$. In fact, it is well known that the Wigner transform may produce the so-called ghost frequencies. As observed in \cite{CGRPartII2022,CRcharModSp}, they are exactly preserved for  Schr\"odinger propagators  for linear $\chi$, i.e., quadratic $\Phi$, but this is not the case for nonlinear $\chi$. Namely, highly oscillating terms may appear in the  expression of the kernel $k(z,w)$ outside the graph of $z=\chi(w)$. 

As a first attempt for eliminating ghost frequencies and re-establishing the validity of \eqref{L3}, we shall consider in the sequel symbols $\sigma$ of low order in Shubin classes \cite{shubin}. Unluckily, this framework does not allow a direct application to Schr\"odinger equations, for which we address a future work, following a different smoothing procedure.

Let us outline the contents of the paper. Our starting point, in Section $3$, will be the following \emph{abstract} definition, in the lines of \cite{CGNRJMPA}.
\begin{definition}\label{C6T1.1}
Consider a \emph{tame}  symplectic diffeomorphism $\chi$ (cf. Definition \ref{def2.1} below).  For  $N\in \bN_+$, $N>d$,  we say that the operator $K$ in \eqref{I3} is in the class FIO($\chi$, $N$) if its Wigner kernel $k$ in \eqref{I4} satisfies, for $z=(z_1,z_2)$, $w=(w_1,w_2)\in \rdd$,
\begin{equation}\label{nucleoFIO}
	|k(z,w) |\lesssim \frac{1}{\la z-\chi(w)\ra^{2N}}.
\end{equation}
\end{definition}
Examples of operators which fall in the above class are pseudodifferential operators $\sigma(x,D)$ (the Kohn-Nirenberg form), defined by
\begin{equation}\label{C6kohnNirenberg}
	\sigma(x,D) f(x)=\intrd e^{2\pi i x \xi}\sigma(x,\xi)\hat{f}(\xi)\,d\xi,
\end{equation}
with a symbol $\sigma$ in the  Shubin classes $\Gamma^m(\rdd)$, $m<-2(d+N)$, whose Wigner kernel $k_\sigma$ satisfies
\begin{equation}\label{nucleoFIOpseudo}
	|k_\sigma(z,w) |\lesssim \frac{1}{\la z-w\ra^{2N}}.
\end{equation}
Here $\chi=I$, the identity mapping, cf. Section $2$ below.  More generally,  Fourier integral operators of type I (cf. \eqref{L2}) and II, having symbols in the same Shubin classes above and \emph{tame} canonical transformations, fall in the class above, as we shall show in Sections 4 and 5.

Let us state here the preliminary results of Section $3$, which are the core of this study and may be collected as follows.
\begin{theorem}[Properties of the class FIO($\chi$, $N$)] \label{tc1} \hspace*{5cm}
	
	(i) \textit{Boundedness}. $T\in  FIO(\chi, N)$ is bounded on $\lrd$.

	(ii)  \textit{Algebra Property}. If $T_i\in FIO(\chi_i,N)$, $i=1,2$, then $T_1 T_2\in
	FIO(\chi_1\chi_2,N)$.
	
	(iii)  If $T\in FIO(\chi, N)$ then its adjoint $T^{*}$ is in $FIO(\chi^{-1}, N)$.
\end{theorem}

In Section $4$ we shall show the Fourier integral operators of type I in \eqref{L2}, having symbols in suitable Shubin classes $\Gamma^m(\rdd)$ and tame phase functions are in the class FIO($\chi$, $N$).  

The last Section $5$ is devoted to the $L^2$-adjoint of the FIO I in \eqref{L2}, which can be written explicitly in the form
$$
T_{II}f(x)=\int_{\rdd}e^{-2\pi i[\Phi(y,\xi)-x\xi]}\tau(y,\xi)f(y)dyd\xi, \qquad f\in\cS(\rd).
$$
Using tools from metaplectic Wigner distributions implemented in \cite{CRGFIO1,CGRPartII2022} we are able to compute the Wigner kernel of the FIOs II above and prove that, under suitable assumptions on their symbols, they belong to $FIO(\chi,N)$ as well.  We underline that these results are valid for the whole class of tame phase $\Phi$ defined in Subsection $2.3$, of particular interest is the case $\Phi$ non quadratic which gives rise to nonlinear symplectic transformations $\chi$, which were not treated in \cite{CRGFIO1}.

We believe that such theoretical study will pave the way to a better understanding of Wigner kernels for Fourier integral operators, with possible applications to dynamical versions of Hardy's uncertainty principles \cite{FM2021,Helge,Z1,Z2,ZH}, see also the recent contribution \cite{CGP2024}.
\section{Preliminaries}\label{sec:preliminaries}
\textbf{Notation.} We define $t^2=t\cdot t$,  $t\in\rd$, and, similarly,
$xy=x\cdot y$.  The space   $\sch(\Ren)$ is the Schwartz class and $\sch'(\Ren)$ its dual (the space of temperate distributions).   The brackets  $\la f,g\ra$ means the extension to $\sch' (\Ren)\times\sch (\Ren)$ of the inner product $\la f,g\ra=\int f(t){\overline {g(t)}}dt$ on $L^2(\Ren)$ (conjugate-linear in the second component).

A point in the phase space is denoted by
$z=(x,\xi)\in\rdd$. We call (\tfs )
the operators
\begin{equation}
	\label{eq:kh25}
	\pi (z)f(t) = e^{2\pi i \xi t} f(t-x), \, \quad t\in\rd.
\end{equation}
 $GL(d,\R)$ denotes the group of real invertible $d\times d$ matrices. 
\subsection{The symplectic group $Sp(d,\mathbb{R})$, metaplectic operators and Wigner distributions}
The standard symplectic matrix is
\begin{equation}\label{J}
	J=\begin{pmatrix} 0_{d\times d}&I_{d\times d}\\-I_{d\times d}&0_{d\times d}\end{pmatrix}.
\end{equation}
The symplectic group is defined by
\begin{equation}\label{defsymplectic}
	\Spnr=\left\{\cA\in\Gltwonr:\;\cA^T J\cA=J\right\},
\end{equation}
where  $\cA^T$ is the transpose of $\cA$. We have  $\det(\cA)=1$.\par
For $L\in GL(d,\bR)$ and $C\in Sym(2d,\bR)$, define:
\begin{equation}\label{defDLVC}
	\cD_L:=\begin{pmatrix}
		L^{-1} & 0_{d\times d}\\
		0_{d\times d} & L^T
	\end{pmatrix} \qquad \text{and} \qquad V_C:=\begin{pmatrix}
		I_{d\times d} & 0\\ C & I_{d\times d}
	\end{pmatrix}.
\end{equation}
The matrices $J$, $V_C$, and $\cD_L$  generate the group $Sp(d,\bR)$.\par

The Schr\"odinger representation $\rho$  of the Heisenberg group is given by $$\rho(x,\xi;\tau)=e^{2\pi i\tau}e^{-\pi i\xi x}\pi(x,\xi),$$ for all $x,\xi\in\rd$, $\tau\in\bR$.
For every $A\in Sp(d,\bR)$, $\rho_A(x,\xi;\tau):=\rho(A (x,\xi);\tau)$ defines another representation of the Heisenberg group that is equivalent to $\rho$, that is, there exists a unitary operator $\hat A:L^2(\rd)\to L^2(\rd)$ such that
\begin{equation}\label{muAdef}
	\hat A\rho(x,\xi;\tau)\hat A^{-1}=\rho(A(x,\xi);\tau), \qquad  x,\xi\in\rd, \ \tau\in\bR.
\end{equation}
This operator is not unique: if $\hat A'$ is another unitary transformation satisfying (\ref{muAdef}), then $\hat A'=c\hat A$, for some  $c\in\bC$, with $|c|=1$. The set $\{\hat A : A\in Sp(d,\bR)\}$ is a group under operator composition and has the metaplectic group $Mp(d,\bR)$ as subgroup. It is a realization of the two-fold cover of $Sp(d,\bR)$. The projection
\begin{equation}\label{piMp}
	\pi^{Mp}:Mp(d,\bR)\to Sp(d,\bR)
\end{equation} is a group homomorphism with kernel $\ker(\pi^{Mp})=\{-id_{{L^2}},id_{{L^2}}\}$.

Here, if $\hat A\in Mp(d,\bR)$, the matrix $A$ will  be the unique symplectic matrix satisfying $\pi^{Mp}(\hat A)=A$.
Some examples of metaplectic operators we will use in  the following are detailed below.
\begin{example}\label{es22} Consider the matrices $J$, $\cD_L$ and $V_C$  defined  in (\ref{J}) and (\ref{defDLVC}). Then, if we denoted by $\cF$ the Fourier transform,
	\begin{enumerate}
		\item[\it (i)] $\pi^{Mp}(\cF)=J$;
		\item[\it (ii)] if $\mathfrak{T}_L:=|\det(L)|^{1/2}\,f(L\cdot)$, then $\pi^{Mp}(\mathfrak{T}_L)=\cD_L$;
		\end{enumerate}
\end{example}
The relation between \tfs s and metaplectic operators is the following:
\begin{equation}\label{metap}
	\pi (\cA z) = c_\cA  \, \hat\cA  \pi (z) \hat\cA\inv  \quad  \forall
	z\in \rdd \, ,
\end{equation}
with a phase factor $c_\cA \in \bC , |c_{\cA }| =1$  (see, e.g., \cite{folland89,Gos11}).\\

\textbf{Metaplectic Wigner distributions.}
In the study of FIOs of type II we will use tools from the theory of metaplectic Wigner distributions. Here we list the basic elements for this study.
For $\hat\cA\in Mp(2d,\bR)$, the \emph{metaplectic Wigner distribution} associated to $\hat\cA$ is defined  as
\begin{equation}\label{WA}
	W_\cA(f,g)=\hat\cA(f\otimes\bar g),\quad f,g\in L^2(\rd).
\end{equation}
The most important time-frequency representations are metaplectic Wigner distributions.
The  $\tau$-Wigner distributions, $\tau\in\bR$, defined by
\begin{equation}\label{tauWigner}
	W_\tau(f,g)(x,\xi)=\int_{\rd} f(x+\tau t)\overline{g(x-(1-\tau)t)}e^{-2\pi i\xi  t}dt, \qquad  (x,\xi)\in\rdd,
\end{equation}
for $f,g\in L^2(\rd)$, are metaplectic Wigner distributions.
The case  $\tau=1/2$ is the cross-Wigner distribution, defined in \eqref{CWD}.
$\tau$-Wigner distributions are metaplectic Wigner distributions:
$$W_\tau(f,g)=\hat A_\tau(f\otimes\bar g),$$ with
\begin{equation}\label{Atau}
	A_\tau=\begin{pmatrix}
		(1-\tau)I_{d\times d} & \tau I_{d\times d} & 0_{d\times d} & 0_{d\times d}\\
		0_{d\times d} & 0_{d\times d} & \tau I_{d\times d} & -(1-\tau)I_{d\times d}\\
		0_{d\times d} & 0_{d\times d} & I_{d\times d} & I_{d\times d}\\
		-I_{d\times d} & I_{d\times d} & 0_{d\times d} & 0_{d\times d}
	\end{pmatrix}.
\end{equation}
In particular, we recapture the Wigner case when $\tau=1/2$:
\begin{equation}\label{A12}
	Wf=W_{1/2}(f,f)=\hat A_{1/2}(f\otimes\bar f), \quad f\in L^2(\rd).
\end{equation}
 $\hat A_{1/2}$ can be split into the product
\begin{equation}\label{A12prodotto}
	\hat A_{1/2}=\cF_2\mathfrak{T}_L,
\end{equation}
with
\[
L=\begin{pmatrix}I_{d\times d} & \frac{1}{2}I_{d\times d}\\ I_{d\times d} & -\frac{1}{2}I_{d\times d}\end{pmatrix}.
\]
Hence
\[
\hat A_{1/2}F(x,\xi)=\int_{\rd}F(x+t/2,x-t/2)e^{-2\pi i\xi t}dt, \qquad F\in \cS(\rdd),
\]
and $$\hat A_{1/2}^{-1}=\mathfrak{T}_{L^{-1}}\cF_2^{-1},$$ where
\[
L^{-1}=\begin{pmatrix} \frac{1}{2}I_{d\times d} & \frac{1}{2}I_{d\times d}\\
	I_{d\times d} & -I_{d\times d}
\end{pmatrix},
\]
so that
\begin{equation}\label{rem1}
	\hat A_{1/2}^{-1}F(x,\xi)=\int_{\rd}F(x/2+\xi/2,y)e^{2\pi i(x-\xi) y}dy, \qquad F\in\cS(\rdd).
\end{equation}
\subsection{Shubin and H\"{o}rmander classes \cite{shubin}, \cite{helffer84}, \cite{Elena-book}}
In our study we shall consider the following weight functions
\begin{equation}\label{vs} v_s(z)=\la z\ra^s=(1+|z|^2)^{\frac s 2},\quad s\in\R,
\end{equation}
\begin{definition}\label{shubinclass}
	Fix $m\in\bR$. The 
	shubin class $\Gamma^m(\rdd)$ is the set of functions $a\in\mathcal{C}^\infty(\rdd)$ satisfying  
	$$|\partial^\a_z a(z)|\leq C_\a v_{m-|\a |}(z),\quad z\in\rdd, \,\a \in \bZ^{2d}_+, 
	$$
for a suitable  constant $C_\a >0$,	where  $v_s(z)=\la z\ra^s$ is defined in \eqref{vs}.
\end{definition}

The H\"{o}rmander class $S^0_{0,0}(\rdd)$, consists of smooth functions  $\sigma$  on $\rdd$ such that 
\begin{equation}\label{I5}
	|\partial_x^\alpha \partial_\xi^\beta \sigma(x,\xi)|\leq c_{\alpha,\beta},\quad\alpha,\beta\in\N^d,\quad x,\xi\in\rd.
\end{equation}

\subsection{Tame phase functions and related canonical transformations}
\begin{definition}\label{def2.1} We follow the notation of \cite{CGNRJMPA,CRGFIO1}. A real  phase function  $\Phi(x,\eta)$  is named \emph{tame} if it satisfies the following
properties:\\
{\it  A1.} $\Phi\in \cC^{\infty}(\rdd)$;\\
{\it  A2.} For $z=\phas\in\rdd$,
\begin{equation}\label{phasedecay}
	|\partial_z^\a \Phi(z)|\leq
	C_\a,\quad |\a|\geq
	2;\end{equation}
{\it  A3.} There exists $\delta>0$:
\begin{equation}\label{detcond}
	|\det\,\partial^2_{x,\eta} \Phi(x,\o)|\geq \delta.
\end{equation}
\par
Solving the system
\begin{equation}\label{cantra} \left\{
	\begin{array}{l}
		y=\Phi_\eta(x,\eta),
		\\
		\xi=\Phi_x(x,\eta), \rule{0mm}{0.55cm}
	\end{array}
	\right.
\end{equation}
with respect to $(x,\xi)$, one obtains a
map $\chi$
 \begin{equation}\label{chi}
	(x,\xi)=\chi(y,\o),
\end{equation}
 with the following properties:\\
 
 \noindent {\it  A4.} $\chi:\rdd\to\rdd$ is a \emph{symplectomorphism} (smooth, invertible,  and
 preserves the symplectic form in $\rdd$, i.e., $dx\wedge d\xi= d
 y\wedge d\eta$.)
 \\
 {\it  A5.} For $z=(y,\eta)$,
 \begin{equation}\label{chistima}
 	|\partial_z^\a \chi(z)|\leq C_\a,\quad |\a|\geq 1;\end{equation}
 {\it A6}. There exists $\delta>0$: 
 \begin{equation}\label{detcond2}
 	|\det\,\frac{\partial x}{\partial y}(y,\eta)|\geq \delta\quad\mbox{for}\,\,
 	(x,\xi)=\chi(y,\eta).
 \end{equation}
\end{definition}
 Conversely,  as it was observed in \cite{CGNRJMPA}, to every transformation $\chi$ satisfying the three hypothesis above corresponds a tame phase $\Phi$, uniquely
 determined up to a constant. 
\subsection{Properties of the Wigner Kernel}
The Wigner kernel of a linear bounded  operator $T:\cS(\rd)\to\cS'(\rd)$ was introduced and studied in \cite{CRGFIO1}. We recall its definition and the properties useful for our framework.
\begin{definition}\label{defWKernel}
  The \emph{Wigner kernel} of  $T:\cS(\rd)\to\cS'(\rd)$  linear and bounded operator is the distribution $k\in\cS'(\bR^{4d})$  satisfying
	\begin{equation}\label{kerFormula}
		\la W(Tf,Tg),W(u,v)\ra=\la k,W(u,v)\otimes\overline{W(f,g)}\ra, \qquad f,g,u,v\in\cS(\rd).
	\end{equation}
\end{definition}
Observe that if $k\in\cS(\bR^{4d})$ the integral formula \eqref{I4} holds true.
The results of Theorem 3.3 and 4.3 in \cite{CRGFIO1} can be rephrased as follows:
\begin{theorem}\label{3.3}
Consider $T$ as above and let $k_T\in\cS'(\rdd)$ be its kernel. There exists a unique distribution $k\in\cS'(\bR^{4d})$ such that \eqref{kerFormula} holds. Hence, every bounded linear operator $T:\cS(\rd)\to\cS'(\rd)$ has a unique Wigner kernel. Furthermore,
	\begin{equation}\label{nuclei}
		k=\mathfrak{T}_pWk_T,
	\end{equation}
	with $\mathfrak{T}_pF(x,\xi,y,\eta)=F(x,y,\xi,-\eta)$.
	
	In particular,
	 if $T\in B(L^2(\rd))$ has Wigner kernel $k$, then its adjoint $T^*\in B(L^2(\rd))$ has Wigner kernel $\tilde k(z,w)= k(w,z)$, $z,w\in\rdd$.
\end{theorem}
\section{Properties of FIO($\chi, N$)}
This section is devoted to prove  Theorem \ref{tc1} in the introduction. This requires several steps, developed in what follows.
\begin{theorem}\label{Pi}
An operator $T\in  FIO(\chi,N)$,  is bounded on $\lrd$.
\end{theorem}
\begin{proof}
	For $f\in \lrd$ we recall \cite[Chapter 1]{Elena-book} that the Wigner $Wf\in\lrdd$ and Moyal's identity $\|Wf\|_{L^2(\rdd)}=\|f\|^2_{L^2(\rd)}$.

	Using \eqref{I4},  Definition \ref{C6T1.1},   for any $f\in\lrd$,
	$$\|Tf\|^2_{L^2(\rd)}=\|W(Tf)\|_{L^2(\rdd)},
	$$
	and
	\begin{align*}
		\|W(Tf)\|_{L^2(\rdd)}&\lesssim \left\|\intrdd\frac{1}{\la z-\chi(w)\ra^{2N}}Wf(w)\,dw\right\|_{L^2(\rdd)}\\
		&\asymp \left\|\intrdd\frac{1}{\la \chi^{-1}(z)-w\ra^{2N}}Wf(w)\,dw\right\|_{L^2(\rdd)}\\
		&\lesssim \left\|\left(\frac{1}{\la\cdot\ra^{2N}}\ast Wf\right)(\chi^{-1}(z))\right\|_{L^2(\rdd)}\\
		&\lesssim \left\|\left(\frac{1}{\la\cdot\ra^{2N}}\ast Wf\right)(z)\right\|_{L^2(\rdd)}\\
		&\leq\left\|\frac{1}{\la\cdot\ra^{2N}}\right\|_{L^1(\rdd)}\|Wf\|_{L^2(\rdd)}\\
		&\leq C_N \|f\|^2_{L^2(\rd)},
	\end{align*}
	where in the last row we used Young's inequality (observe that $N>d$) and, in the last but one,   the change of variables $z'=\chi^{-1}(z)$ which, for any $F\in L^2(\rdd)$, 
	\begin{align*}
		\|F(\chi^{-1}\cdot)\|^2_{L^2(\rdd)}&=\int_{\rdd}|F(\chi^{-1}(z))|^2\,dz=\int_{\rdd}|F(z')|^2 \det |J\chi(z')|dz'\\
		&\leq C_N\|F\|^2_{L^2(\rdd)},
	\end{align*}
by \eqref{chistima}.
	Hence $\|Tf\|_{L^2(\rd)}\leq \sqrt{C_N} \|f\|_{L^2(\rd)}$, that is $T\in B(L^2(\rd))$.
\end{proof}	 
	
	\begin{theorem}\label{Pii}
		If $T_i\in FIO(\chi_i,N)$, $i=1,2$,   then $T_1 T_2\in
		FIO(\chi_1\chi_2,N)$.
	\end{theorem}
	\begin{proof}
		Using the Wigner representation in \eqref{I4} we can write 
		\begin{equation}\label{I4iie}
			W(T_1T_2f,T_1T_2g)(z) = \intrdd k_{I,1}(z,w) W(T_2f,T_2g)(w)\,dw,
		\end{equation}
	where $k_{I,1}(z,w)$ is the Wigner kernel of the operator $T_1$, satisfying \eqref{nucleoFIO} with symplectic transformation $\chi_1$. Similarly, 
	\begin{equation*}\label{I4ii}
		W(T_2f,T_2g)(w) = \intrdd k_{I,2}(w,u) W(f,g)(u)\,du,
	\end{equation*}
with $k_{I,2}(w,u)$ being the Wigner kernel of $T_2$ satisfying \eqref{nucleoFIO} with symplectic transformation $\chi_2$. Substituting the expression of $W(T_2f,T_2g)(w)$ in \eqref{I4iie} we obtain
\begin{equation}\label{e0}
	W(T_1T_2f,T_1T_2g)(z) = \int_{\bR^{4d}} k_{I,1}(z,w)  k_{I,2}(w,u) W(f,g)(u)\,du\,dw,
\end{equation}
with $k_{I,i}(z,w)$ satisfying   \eqref{nucleoFIO}, $i=1,2$. 
 Interchanging the integrals in \eqref{e0} (observe that the assumptions of Fubini Theorem are satisfied) we can write
\begin{equation}\label{e00}
	W(T_1T_2f,T_1T_2g)(z) = \int_{\bR^{2d}}\left(\intrdd k_{I,1}(z,w)  k_{I,2}(w,u)\,dw\right) W(f,g)(u)\,du
\end{equation}
so that the Wigner kernel $k_{I,1 2}$ of the product $T_1T_2$ is given by
$$k_{I,1 2}(z,u):= \intrdd k_{I,1}(z,w)  k_{I,2}(w,u)\,dw.
$$
Using the Wigner kernel's estimates in \eqref{nucleoFIO},
we reckon
\begin{align*}
|k_{I,1 2}(z,u)|&\leq \intrdd| k_{I,1}(z,w)|  | k_{I,2}(w,u)|dw\\
&\lesssim\intrdd \frac{1}{\la z-\chi_1(w)\ra^{2N}\la w-\chi_2(u)\ra^{2N}} dw\\
&\asymp \intrdd \frac{1}{\la z-\chi_1(w)\ra^{2N}\la \chi_1(w)-\chi_1\chi_2(u)\ra^{2N}}dw\\
&\lesssim\intrdd  \frac{1}{\la z-\chi_1(w)\ra^{2N}\la \chi_1(w)-\chi_1\chi_2(u)\ra^{2N}}dw\\
&=\intrdd  \frac{1}{\la \chi_1(w)-z\ra^{2N}\la \chi_1\chi_2(u)-\chi_1(w)\ra^{2N}}dw\\
&=\intrdd  \frac{1}{\la w'-z\ra^{2N}\la \chi_1\chi_2(u)-w'\ra^{2N}}|\det J\chi_1^{-1}(w)| dw'\\
\end{align*}
where we used the change of variables
$\chi_1(w)=w'$ so that $dw=|\det J\chi_1^{-1}(w)| dw'$
since $|\det J\chi_1^{-1}(w)|\leq C$ by \eqref{chistima}, we obtain
\begin{align*}
	|k_{I,1 2}(z,u)|&\lesssim \intrdd  \frac{1}{\la w'-z\ra^{2N}\la \chi_1\chi_2(u)-w'\ra^{2N}}dw'\\
	&=\intrdd \frac{1}{\la v \ra^{2N}\la \chi_1\chi_2(u)-z-v \ra^{2N}}dv\\
&=(\la \cdot\ra^{-2N}\ast \la \cdot\ra^{-2N})(\chi_1\chi_2(u)-z)\\
&\lesssim \frac{1}{\la z-\chi_1\chi_2(u)\ra^{2N}}
\end{align*}
where in the last row we used the weight convolution property $\la\cdot\ra^s\ast \la\cdot\ra^s\lesssim \la\cdot\ra^s$ for $s<-2d$ (observe $N>d$).
Thus, we obtain the desired estimate
$$|k_{I,1 2}(z,u)|\lesssim  \frac{1}{\la z-\chi_1\chi_2(u)\ra^{2N}},$$
that is $T_1T_2\in FIO(\chi_1\chi_2,N)$.
\end{proof}
	\begin{theorem}\label{adjoint}
		If $T\in FIO(\chi,N)$,  
		then $T^{*}\in FIO(\chi^{-1},N)$.
	\end{theorem}
\begin{proof}
	Theorem \ref{Pi} gives that $T\in B(L^2(\rd))$. Let $k$ be integral kernel of $T$, then Theorem \ref{3.3} says that the adjoint $T^*\in B(L^2(\rd))$ has kernel $\tilde k$ given by
	$$\tilde k(z,w)={k(w,z)}.$$ This means it satisfies (cf.  \eqref{nucleoFIO})
	\begin{equation}\label{nucleoFIOagg}
		|\tilde k(z,w) |=| k(w,z) |\lesssim \frac{1}{\la w-\chi(z)\ra^{2N}}.
	\end{equation}
Since $\chi$ is a bi-Lipschitz transformation, $|w-\chi(z)|\asymp |z-\chi^{-1}(w)|$ so that 
$\la w-\chi(z)\ra^{2N}\asymp \la z-\chi^{-1}(w)\ra^{2N}$ and we obtain
\begin{equation}
	|\tilde k(z,w) |\lesssim \frac{1}{ \la z-\chi^{-1}(w)\ra^{2N}}.
\end{equation}
Hence $T^*\in  FIO(\chi^{-1},N)$, as desired.
\end{proof}
\section{FIOs of type I}
Here we focus on the analysis of  Wigner kernels for FIOs of type I:
\begin{equation}\label{tipo1}
	T_If(x)=\int_{\rd}e^{2\pi i\Phi(x,\xi)}\sigma(x,\xi)\hat f(\xi)d\xi, \qquad f\in\cS(\rd).
\end{equation}
Recall that he Schwartz' Kernel Theorem guarantees that every continuous linear
operator $T:\cS(\rd)\to\cS'(\rd)$ can be expressed in the form $T=T_I$,  with a given phase $\Phi(x,\xi)$ and  symbol $\sigma(x,\xi)$ in $\cS'(\rdd)$. 

These operators have been widely investigated in the framework of PDEs, both from a theoretical and a numerical point of view; the literature is so huge that we cannot report all the results but limit to a very partial list of them, cf. \cite{CD2007,wiener8,wiener9,fio1,fio3,B18,B36}.

If we assume that $T$ is a continuous
linear operator $\cS(\rd)\to\cS'(\rd)$ and $\chi$ satisfies conditions {\it
	A4}, {\it A5}, and {\it A6} in Definition \ref{def2.1} then $T=T_{I,\Phi_\chi,\sigma}$  (FIO
of type I), with symbol $\sigma$ and phase $\Phi_\chi$.

A first result related  to the Wigner kernel of a FIO I  was obtained in   \cite[Theorem 5.8]{CRGFIO1}. There, FIOs of type I with  symbols in the H\"ormander class $S^0_{0,0}(\rdd)$ were considered. Since $\Gamma^{m}(\rdd)\subset  S^0_{0,0}(\rdd)$ whenever $m\leq0$, we can   rephrased it  in our context as follows.
\begin{theorem}\label{teor3.6}
Let $T_I$ be a FIO of type I defined in \eqref{tipo1} with symbol $\sigma\in\Gamma^{m}(\rdd)$, $m<0$. For $f\in\cS(\rd)$, 
	\begin{equation}\label{K}
		K(W(f,g))(x,\xi)=W(T_If,T_Ig)(x,\xi)=\int_{\rdd}k_I(x,\xi,y,\eta)W(f,g)(y,\eta)dyd\eta,
	\end{equation}
	with  Wigner kernel $k_I$ given by
	\begin{equation}\label{KI}
		k_I(x,\xi,y,\eta)=\int_{\rdd}e^{2\pi i[\Phi_I(x,\eta,t,r)-(\xi t+ry)]} \sigma_I(x,\eta,t,r)dtdr,
	\end{equation}
	and, for $x,\eta,t,r\in\rd$,
	\begin{equation}\label{fii}
		\Phi_I(x,\eta,t,r)=\Phi(x+\frac{t}{2},\eta+\frac{r}{2})-\Phi(x-\frac{t}{2},\eta-\frac{r}{2}),
	\end{equation}
	\begin{equation}\label{sigmai}
		\sigma_I(x,\eta,t,r):=\sigma(x+\frac{t}{2},\eta+\frac{r}{2})\overline{\sigma(x-\frac{t}{2},\eta-\frac{r}{2})}.
	\end{equation}
\end{theorem}

We have now all the tools to estimate the Wigner kernel of $T_I$.

\begin{theorem}\label{T4.1}
	Consider $T_I$ the FIO of type I in \eqref{tipo1}. Fix $N\in \bN$, $N>d$, and assume that the symbol $\sigma\in \Gamma^{m}(\rdd)$, with $m<-2(d+N)$. Let $k_I$ be the associated Wigner kernel in \eqref{KI}. Then,
	\begin{equation}\label{nuclei-trasf}
		|k_I(x,\xi,y,\eta)|\lesssim\frac{\la (x,\eta)\ra^{2N+m}}{\la (x,\xi)-\chi(y,\eta)\ra^{2N}},\qquad x,\xi,y,\eta\in\rdd.
	\end{equation}  
\end{theorem}
\begin{proof}
	Since $\Phi$ is smooth, we can expand $\Phi(x+\frac{t}{2},\eta+\frac{r}{2})$ and $\Phi(x-\frac{t}{2},\eta-\frac{r}{2})$
	into a Taylor series around
	$\phas$. Namely,
	\begin{equation}\label{E1}
		\Phi\left(x+\frac{t}{2},\eta+\frac{r}{2}\right)=\Phi(x,\eta)+\frac {t}2\Phi_x (x,\eta)+\frac {r}2 \Phi_\eta (x,\eta) +\Phi_{2}(x,\eta,t,r),
	\end{equation}
	where the remainder $\Phi_{2}$ is given by
	\begin{equation}
		\label{eq:c11}
		\Phi_{2}(x,\eta,t,r)=\sum_{|\a|=2}\int_0^1(1-\tau)\partial^\a
		\Phi((x,\eta)+\tau(t,r)/2)\,d\tau\frac{(t,r)^\a}{2^3\a!}.
	\end{equation}
	
	Similarly,

	\begin{equation}\label{E2}
		\Phi\left(x-\frac{t}{2},\eta-\frac{r}{2}\right)=\Phi(x,\eta)-\frac {t}2\Phi_x (x,\eta)-\frac {r}2 \Phi_\eta (x,\eta) +\widetilde{\Phi}_{2}(x,\eta,t,r),
	\end{equation}
	with
	$\widetilde{\Phi}_{2}$ defined as
	\begin{equation}
		\label{eq:c12}
		\widetilde{\Phi}_{2}(x,\eta,t,r)=\sum_{|\a|=2}\int_0^1(1-\tau)\partial^\a
		\Phi((x,\eta)-\tau(t,r)/2)\,d\tau\frac{(t,r)^\a}{2^3\a!}.
	\end{equation}
	Inserting the phase expansions above in \eqref{KI} we obtain
	\begin{equation}\label{KI-exp}
		k_I(x,\xi,y,\eta)=\int_{\rdd} e^{-2\pi i[t\cdot(\xi-\Phi_x(x,\eta))+ r\cdot(y-\Phi_\eta(x,\eta))]}\tilde{\sigma}(x,\eta,t,r)\,dtdr
	\end{equation}
	where
	$\tilde{\sigma}$ is defined as
	\begin{equation}\label{sigmatilde}
		\tilde{\sigma}(x,\eta,t,r)=e^{2\pi
			i[\Phi_{2}-\widetilde{\Phi}_{2}](x,\eta,t,r)} \sigma(x+\frac{t}{2},\eta+\frac{r}{2})\overline{\sigma(x-\frac{t}{2},\eta-\frac{r}{2})},
	\end{equation}
	
	For $N\in\bN$, $u=(t,r)\in\rdd$, using the identity:
	\begin{multline*}(1-\Delta_u)^Ne^{-2\pi i[(\xi-\Phi_x(x,\eta), y-\Phi_\eta(x,\eta))\cdot (t,r)]}\\
		=\la
		2\pi(\xi-\Phi_x(x,\eta), y-\Phi_\eta(x,\eta))\ra^{2N} e^{-2\pi i [(\xi-\Phi_x(x,\eta), y-\Phi_\eta(x,\eta))\cdot (t,r)]},
	\end{multline*}
	we integrate by parts in \eqref{KI-exp} and obtain
	\begin{align*}
		k_I(x,\xi,y,\eta)&=\frac1{\la
			2\pi(\xi-\Phi_x(x,\eta), y-\Phi_\eta(x,\eta))\ra^{2N}}\int_{\rdd} e^{-2\pi i [(\xi-\Phi_x(x,\eta), y-\Phi_\eta(x,\eta))\cdot (t,r)]} \\
		&\qquad\qquad\times (1-\Delta_u)^N \tilde{\sigma}(x,\eta,t,r)\,dtdr.\\
	\end{align*}
	The
	factor  $$(1-\Delta_u)^N \tilde{\sigma}(z,u),\quad z=(x,\eta),\,\,u=(t,r)$$
	can be expressed as
	$$e^{2\pi
		i[\Phi_{2}-\widetilde{\Phi}_{2}](z,u)}\sum_{|\a|+|\b|+|\gamma|\leq 2N}
	C_{\a,\b\gamma}p( \partial_u^{|\a|}(\Phi_{2}-\widetilde{\Phi}_{2})_z(u)
	(\partial_u^\b
	\sigma)(z+u/2)(\partial_u^\gamma\sigma)(z-u/2),$$
	where $p(\partial_u^{|\a|} (\Phi_{2}-\widetilde{\Phi}_{2})_z)(u)$ is a   polynomial  made  of derivatives w.r.t. $u$ of
	$\Phi_{2}-\widetilde{\Phi}_{2}$ of order at most $|\a|$. By assumption,
	$$|(\partial_u^\b
	\sigma)(z+u/2)(\partial_u^\gamma\sigma)(z-u/2)|\lesssim \la z+u/2\ra^{m-|\beta|} \la z-u/2\ra^{m-|\gamma|},$$
	which implies 
	\begin{align*}|(1-\Delta_u)^N \tilde{\sigma}(z,u)|&\lesssim \sum_{|\a|+|\b|+|\gamma|\leq 2N}
		\la u/2\ra^{|\a|}\la z-u/2\ra^{m-|\beta|} \la z+u/2\ra^{m-|\gamma|}\\
		&\lesssim  \sum_{|\b|+|\gamma|\leq 2N}
		\la u/2\ra^{2N-|\beta|-|\gamma|}\la z-u/2\ra^{m-|\beta|} \la z+u/2\ra^{m-|\gamma|}\\
		&\lesssim  C_N\la z-u/2\ra^{2N+m} \la z+u/2\ra^{2N+m}.
	\end{align*}
	Using the change of variables $u'=u/2-z$, $du=2^{2d}du'$,
	$$\intrdd \la z-u/2\ra^{2N+m} \la z+u/2\ra^{2N+m}du=2^{2d}\intrdd \la u'\ra^{2N+m}\la (-2z)-u'\ra^{2N+m}\lesssim \la z\ra^{2N+m}$$
	where, for $v_s=\la \cdot\ra^s$,  we used the weight convolution property $v_s\ast v_s\lesssim v_s$, for $s<2d$, cf. \cite[Lemma 11.1.1]{book}. Hence,
	\begin{align*}
		|k_I(x,\xi,y,\eta)|&\leq \frac1{\la
			2\pi(\xi-\Phi_x(x,\eta), y-\Phi_\eta(x,\eta))\ra^{2N}}\int_{\rdd} |(1-\Delta_u)^N \tilde{\sigma}(x,\eta,t,r)|\,dtdr\\
		&\lesssim \frac{\la z\ra^{2N+m}}{\la
			2\pi(\xi-\Phi_x(x,\eta), y-\Phi_\eta(x,\eta))\ra^{2N}}\\
		&\asymp  \frac{\la z\ra^{2N+m}}{\la
			\chi_1(y,\eta)-x, \chi_2(y,\eta)-\xi\ra^{2N}}.
	\end{align*}
	This gives the claim.
\end{proof}

As a consequence,
\begin{corollary}\label{cor42}
Under the assumptions of Theorem \ref{T4.1},   the estimate \eqref{nucleoFIO}
holds true, hence $T_I\in FIO(\chi,N)$.
\end{corollary}
\begin{proof}
	It is an immediate consequence of Theorem \ref{T4.1}, since $2N+m<0$ so that $\la (x,\eta)\ra^{2N+m}\leq 1$, for every $x,\eta\in\rd$.
\end{proof}

\section{FIOs of Type II}
In this section we focus on the $L^2$-adjoint of a FIO of type $I$, which is a FIO of type II,	 written  formally as
\begin{equation}\label{FIOII}
T_{II}f(x)=\int_{\rdd}e^{-2\pi i[\Phi(y,\xi)-x\xi]}\tau(y,\xi)f(y)dyd\xi, \qquad f\in\cS(\rd).
\end{equation}
First, we shall work with symbols $\tau$ in the  H\"ormander class $S^0_{0,0}(\rdd)$, referring to \cite{AS78} for their $L^2$-boundedness.\par

\begin{proposition}\label{prop3}
Consider a FIO of type II as in \eqref{FIOII}, with $\tau\in S^0_{0,0}(\rdd)$.
	Then, for all $f,g\in\cS(\rd)$,
	\[
		(T_{II}f\otimes \bar g)(x_1,x_2)=T_2(f\otimes \bar g)(x_1,x_2),
	\]
	 $x=(x_1,x_2)\in\rdd$, where $T_2$ is the FIO of type II given by
	\[
		T_2F(x)=\int_{\bR^{4d}}e^{-2\pi i[\Phi_2(y,\xi)-x\xi]}\tau_2(y,\xi)F(y)dyd\xi, \qquad F\in\cS(\rdd),
	\]
 $y=(y_1,y_2)$, $\xi=(\xi_1,\xi_2)\in\rdd$, and $\Phi_2$ is the tame phase on $\bR^{4d}$ given by
	\begin{equation}\label{Pi2}
		\Phi_2(y,\xi)=\Phi(y_1,\xi_1)+y_2\xi_2;
	\end{equation}
whereas 	the symbol $\tau_2$ is in  $S^{0}_{0,0}(\bR^{4d})$ and  given by 
\begin{equation}\label{tau2}
\tau_2(y,\xi)=\tau(y_1,\xi_1)\otimes1(y_2,\xi_2).
	\end{equation}
\end{proposition}
\begin{proof}
	Let $f,g$ be in $\cS(\rd)$. Using  the Fourier  inversion formula on $g$:
	\[
		g(x_2)=\int_{\rdd}g(y_2)e^{2\pi i(x_2-y_2)\xi_2}dy_2d\xi_2
	\]
	we can write
	\[
		\begin{split}
			(T_{II}f\otimes\bar g)(x_1,x_2)&=\Big(\int_{\rdd}e^{-2\pi i[\Phi(y_1,\xi_1)-x_1\xi_1]}\tau(y_1,\xi_1)f(y_1)dy_1d\xi_1\Big)\overline{g(x_2)}\\
			&=\int_{\bR^{4d}}e^{-2\pi i[\Phi(y_1,\xi_1)+y_2\xi_2-(x_1\xi_1+x_2\xi_2)]}f(y_1)\overline{g(y_2)}\tau(y_1,\xi_1)dy_1dy_2d\xi_1d\xi_2.
		\end{split}
	\]
Observing that
	\[
		x_1\xi_1+x_2\xi_2=x \xi,
	\]
	we can write
	\[
		\Phi(y_1,\xi_1)+y_2\xi_2=\Phi_2(y_1,y_2,\xi_1,\xi_2),
	\]
	which is \eqref{Pi2}.  Note that $\Phi_2 \in\cC^{\infty}(\bR^{4d})$ and satisfies \eqref{phasedecay} and \eqref{detcond}, since 
	\begin{equation}\label{detPi2}
		\partial^2_{y,\xi}\Phi_2=\begin{pmatrix}
		\partial^2_{y_1,y_1}\Phi& 0_{d\times d} & \partial^2_{y_1,\xi_1}\Phi & 0_{d\times d} \\
		0_{d\times d} & 0_{d\times d} & 0_{d\times d} & I_{d\times d}\\
		\partial^2_{y_1,\xi_1}\Phi & 0_{d\times d} & \partial^2_{\xi_1,\xi_1}\Phi & 0_{d\times d}\\
		0_{d\times d} &I_{d\times d} & 0_{d\times d} & 0_{d\times d}
		\end{pmatrix}.
	\end{equation}
	This means that $\Phi_2$ is a tame phase function.
	
	Finally, since both $\tau$ and the function constantly equal to $1$ are in the H\"ormander class $S^{0}_{0,0}(\rdd)$, it immediately follows that $\tau_2$ belongs to $S^{0}_{0,0}(\bR^{4d})$.
\end{proof}

Using the same arguments as in the previous proposition, one can prove the issue below.
\begin{proposition}\label{prop4}
	Under the same assumptions of Proposition \ref{prop3}, for all $f,g\in\cS(\rd)$,
	\[
		(f\otimes \overline{T_{II}g})(x_1,x_2)=T_2'(f\otimes \bar g)(x_1,x_2),
	\]
	$x=(x_1,x_2)\in\rdd$, where the operator $T_2'$ is the FIO of type II:
	\[
		T_2'F(x)=\int_{\bR^{4d}}e^{-2\pi i[\Phi_2'(y,\xi)-x\xi]}\tau_2'(y,\xi)F(y)dyd\xi, \qquad F\in\cS(\rdd),
	\]
 $y=(y_1,y_2)$, $\xi=(\xi_1,\xi_2)\in\rdd$ and $\Phi_2'$ is the tame phase
	\begin{equation*}
		\Phi_2'(y,\xi)=-\Phi(y_2,-\xi_2)+y_1\xi_1,
	\end{equation*}
whereas the symbol $\tau_2\in S^{0}_{0,0}(\bR^{4d})$ is given by
$$\tau_2'(y,\xi)=1(y_1,\xi_1)\otimes\overline{\tau(y_2,-\xi_2)}.$$
\end{proposition}
Next, we study the composition of the FIOs $T_2$ and $T_2'$.
\begin{proposition}\label{lemma21}
	Consider the FIOs $T_2$ and $T_2'$ defined in Propositions \ref{prop3} and \ref{prop4}. Then their product $T_2T_2'$ can be written as the following FIO of type II:
	\begin{equation}\label{T2T2'}
		T_2T_2'F(x)=\int_{\bR^{4d}}e^{-2\pi i{\mathbf \Phi}(y,\xi)-x\xi}
		 \Tau(y,\xi) F(y)\,dyd\xi,\quad x\in\rdd,
	\end{equation}
for every $F\in\cS(\rdd)$,
with tame phase on $\bR^{4d}$:
\begin{equation}\label{fiprodotto}
	{\mathbf \Phi}(y_1,y_2,\xi_1,\xi_2)=\Phi(y_1,\xi_1)-\Phi(y_2,-\xi_2)
\end{equation}
and symbol
\begin{equation}\label{tauprodotto}
	\mathbf{\Tau}(y_1,y_2,\xi_1,\xi_2)=\tau(y_1,\xi_1)\overline{\tau(y_2,-\xi_2)}\in S^0_{0,0}(\bR^{4d}).
\end{equation}
\end{proposition}
\begin{proof}
	Let $F\in \cS(\rdd)$. We compute
	\begin{align*}
		T_2T_2'&F(x_1,x_2)=\int_{\bR^{4d}}e^{-2\pi i[\Phi(y_1,\xi_1)+y_2\xi_2-x_1\xi_1-x_2\xi_2]}\tau(y_1,\xi_1)T_2'F(y_1,y_2)dy_1dy_2d\xi_1d\xi_2\\
		&=\int_{\bR^{4d}}e^{-2\pi i[\Phi(y_1,\xi_1)+y_2\xi_2-x_1\xi_1-x_2\xi_2]}\tau(y_1,\xi_1)\int_{\bR^{4d}}e^{-2\pi i[-\Phi(z_2,-\eta_2)+z_1\eta_1-y_1\eta_1-y_2\eta_2]}\\
		& \qquad \times\,\overline{\tau(z_2,-\eta_2)}F(z_1,z_2)dz_1dz_2d\eta_1d\eta_2dy_1dy_2d\xi_1d\xi_2\\
		&=\int_{\bR^{8d}}e^{-2\pi i[\Phi(y_1,\xi_1)-\Phi(z_2,-\eta_2)+z_1\eta_1+y_2\xi_2-x_1\xi_1-x_2\xi_2-y_1\eta_1-y_2\eta_2]}\\
			& \qquad \times\,\tau(y_1,\xi_1)\overline{\tau(z_2,-\eta_2)}F(z_1,z_2)dzd\eta dy dx\xi.
	\end{align*}
	Using the well-known formulae
	\[
		\int_{\rdd}e^{-2\pi i\eta_1(z_1-y_1)}e^{-2\pi i\xi_2(y_2-x_2)}d\eta_1d\xi_2dz_1dy_2=\delta_{y_1}(z_1)dz_1\delta_{x_2}(y_2)dy_2,
	\]
we obtain
\begin{align*}
		T_2T_2'F(x_1,x_2)&=\int_{\bR^{4d}}e^{-2\pi i[\Phi(y_1,\xi_1)-\Phi(z_2,-\eta_2)-x_1\xi_1-x_2\eta_2]}\tau(y_1,\xi_1)\overline{\tau(z_2,-\eta_2)}\\
		&\qquad\quad\times\,F(y_1,z_2)dy_1dz_2d\xi_1d\eta_2,
\end{align*}
which is \eqref{T2T2'}. It is straightforward to check that the phase $	{\mathbf \Phi}$ in \eqref{fiprodotto} is tame, that is, it satisfies the properties of Definition \ref{def2.1}. 

Since $\tau\in S^{0}_{0,0}(\rd)$, then $\Tau\in  S^{0}_{0,0}(\bR^{4d})$.
\end{proof}

\begin{theorem}\label{thm22}
Consider the type II FIO  $T_{II}$ in \eqref{FIOII}, with symbol $\tau\in S^0_{0,0}(\rdd)$ and tame phase $\Phi$. Then  
	\[
	W(T_{II}f,T_{II}g)(x,\xi)=\int_{\rdd}k_{II}(x,\xi,s,z)W(f,g)(s,z)dsdz,\quad f,g\in\cS(\rd),
	\]
	where
	\begin{align}\label{KII}
	k_{II}(x,\xi,y,\eta)&=\int_{\rdd}e^{-2\pi i[\Phi(y+\frac r2,\xi+\frac t2)-\Phi(y-\frac r2,\xi-\frac t2)]}  e^{2\pi i(tx+r\eta)}\\
	&\qquad\times \, \tau(y+\frac r2,\xi+\frac t2)\overline{\tau(y-\frac r2,\xi-\frac t2)}\notag
	dt dr.
\end{align}
\end{theorem}
\begin{proof}
	Consider $f\in\cS(\rd)$ and use Proposition \ref{lemma21} and Remark \ref{rem1} to compute
	\begin{align*}
		W(T_{II}f,T_{II}g)&(x,\xi)=\hat A_{1/2}T_2T_2'\hat A_{1/2}^{-1}W(f,g)(x,\xi)\\
		&=\int_{\rd}(T_2T_2'\hat A_{1/2}^{-1}W(f,g))(x+\frac{t}{2},x-\frac{t}{2})e^{-2\pi i\xi t}dt\\
		&=\int_{\rd}\left(\int_{\bR^{4d}}e^{-2\pi i[\Phi(y_1,\eta_1)-\Phi(y_2,-\eta_2)-(x+t/2)\eta_1-(x-t/2)\eta_2]}\tau(y_1,\eta_1)\overline{\tau(y_2,-\eta_2)}\right.\\
		& \qquad \times\,\left. (\hat A_{1/2}^{-1}W(f,g))(y_1,y_2)dy_1dy_2d\eta_1d\eta_2\right)e^{-2\pi i\xi t}dt\\
		&=\int_{\rd}\left(\int_{\bR^{4d}}e^{-2\pi i[\Phi(y_1,\eta_1)-\Phi(y_2,-\eta_2)-(x+t/2)\eta_1-(x-t/2)\eta_2]}\tau(y_1,\eta_1)\overline{\tau(y_2,-\eta_2)}\right.\\
		&\qquad\times \left. \left(\int_{\rd}W(f,g)(y_1/2+y_2/2,z)e^{2\pi i(y_1-y_2)z}dz\right)dy_1dy_2d\eta_1d\eta_2\right)e^{-2\pi i\xi t}dt\\
		&=\int_{\bR^{6d}}e^{-2\pi i[\Phi(y_1,\eta_1)-\Phi(y_2,-\eta_2)-x\eta_1-\frac{t}{2}\eta_1-x\eta_2+\frac{t}{2}\eta_2-y_1z+y_2z+\xi t]}\tau(y_1,\eta_1)\overline{\tau(y_2,-\eta_2)}\\
		&\qquad \times \,W(f,g)(y_1/2+y_2/2,z)dzdy_1dy_2d\eta_1d\eta_2dt.
	\end{align*}
	The change of variables  $y_1/2+y_2/2=s$ gives
	\begin{align*}
		W(T_{II}f,T_{II}g)(x,\xi)&=2^d\int_{\bR^{6d}}e^{-2\pi i[\Phi(2s-y_2,\eta_1)-\Phi(y_2,-\eta_2)-x\eta_1-\frac{t}{2}\eta_1-x\eta_2+\frac{t}{2}\eta_2-(2s-y_2)z+y_2z+\xi t]}\\
		& \qquad\quad \times \tau(2s-y_2,\eta_1)\overline{\tau(y_2,-\eta_2)}W(f,g)(s,z)dzdsdy_2d\eta_1d\eta_2dt.
	\end{align*}
	Next, observing that
\begin{align*}
		\int_{\rd} e^{-2\pi i(\frac{\eta_2}{2}-\frac{\eta_1}{2}+\xi)t}dt& =\int_{\rd} e^{-2\pi i(\eta_2-\eta_1+2\xi)\frac{t}{2}}dt  =2^{d}\int_{\rd} e^{-2\pi i(\eta_2-\eta_1+2\xi)t'}dt',\\
		&=2^d \int_{\rd} e^{-2\pi i\eta_2 t'}M_{\eta_1-2\xi}1(\eta_2)dt'=2^d T_{\eta_1-2\xi}\hat{1}(\eta_2)\\
		&=2^d T_{\eta_1-2\xi}\delta(\eta_2).
	\end{align*}
we reckon
	\begin{align*}
		W(T_{II}f,T_{II}g)(x,\xi)&=2^{2d}\int_{\bR^{4d}}e^{-2\pi i[\Phi(2s-y_2,\eta_1)-\Phi(y_2,2\xi-\eta_1)+2(\xi-\eta_1)x+2(y_2-s)z]}\\
		&\qquad\quad\times \, \tau(2s-y_2,\eta_1)\overline{\tau(y_2,2\xi-\eta_1)}W(f,g)(s,z)dzdsdy_2d\eta_1\\
		&=\int_{\rdd}k_{II}(x,\xi,s,z)W(f,g)(s,z)dsdz,
	\end{align*}
	where
	\begin{align*}
		k_{II}(x,\xi,s,z)&=2^{2d}\int_{\rdd}e^{-2\pi i[\Phi(2s-y_2,\eta_1)-\Phi(y_2,2\xi-\eta_1)+2(\xi-\eta_1)x+2(y_2-s)z]}\\
		&\qquad\quad\times \, \tau(2s-y_2,\eta_1)\overline{\tau(y_2,2\xi-\eta_1)}
		dy_2d\eta_1.
	\end{align*}
Next, we make the change of variables $s-y_2=r/2$ and $\xi-\eta_1=-t/2$ so that 
	\begin{align*}
	k_{II}(x,\xi,s,z)&=\int_{\rdd}e^{-2\pi i[\Phi(s+\frac r2,\xi+\frac t2)-\Phi(s-\frac r2,\xi-\frac t2)]}  e^{2\pi i(tx+rz)}\\
	&\qquad\times \, \tau(s+\frac r2,\xi+\frac t2)\overline{\tau(s-\frac r2,\xi-\frac t2)}
dt dr,
\end{align*}
which is \eqref{KII}.
\end{proof}

\begin{theorem}\label{stimeII}
	Consider $T_{II}$ the FIO of type II in (\ref{FIOII}). Fix $N\in\mathbb{N}$, $N>d$, and assume that the symbol $\tau\in \Gamma^m(\rdd)$, with $m<-2(d+N)$. Let $k_{II}$ be the associated Wigner kernel, given by \eqref{KII}. Then,
	\begin{equation}\label{50}
		|k_{II}(x,\xi,y,\eta)|\lesssim \frac{\la (y,\xi) \ra^{2N+m}}{\la (y,\eta)-\chi(x,\xi) \ra^{2N}}, \qquad x,\xi,y,\eta\in \rd.
	\end{equation}
\end{theorem}
\begin{proof}
	We follow the pattern of the proof of Theorem \ref{T4.1}. By (\ref{KII}) and using the Taylor expansions in \eqref{eq:c11} and \eqref{eq:c12} we reckon
	\begin{equation}\label{expkII}
	\begin{split}
		k_{II}(x,\xi,y,\eta)&=\int_{\rdd}e^{2\pi i[r\cdot(\eta-\Phi_y(y,\xi))+t\cdot(x-\Phi_\xi(y,\xi))]}\tilde\tau(y,\xi,r,t)drdt,
	\end{split}
	\end{equation}
	where 
	\[
		\tilde\tau(y,\xi,r,t)=e^{-2\pi i[\Phi_2-\tilde\Phi_2](y,\xi,r,t)}\times\tau(y+\frac{r}{2},\xi+\frac{t}{2})\overline{\tau(y-\frac{r}{2},\xi-\frac{t}{2})},
	\]
	and the reminders are given by:
	\[
		\Phi_2(y,\xi,r,t)=\sum_{|\alpha|=2}\int_0^1(1-\tau)\partial^\alpha\Phi((y,\xi)+\tau(r,t)/2)d\tau\frac{(r,t)^\alpha}{2^3\alpha!}
	\]
	and
	\[
		\tilde\Phi_2(y,\xi,r,t)=\sum_{|\alpha|=2}\int_0^1(1-\tau)\partial^\alpha\Phi((y,\xi)-\tau(r,t)/2)d\tau\frac{(r,t)^\alpha}{2^3\alpha!}.
	\]
	Again, for $N\in\mathbb{N}$ and setting $u=(r,t)\in\rdd$, we have:
	\begin{align*}
		(1-\Delta_u)^N&e^{2\pi i(\eta-\Phi_y(y,\xi),x-\Phi_\xi(y,\xi))\cdot(r,t)}\\
		&=\la 2\pi (\eta-\Phi_y(y,\xi),x-\Phi_\xi(y,\xi)) \ra^{2N}e^{2\pi i(\eta-\Phi_y(y,\xi),x-\Phi_\xi(y,\xi))\cdot(r,t)}.
	\end{align*}
	Integrating by parts in (\ref{expkII}), we get:
	\begin{align*}
		k_{II}(x,\xi,y,\eta)&=\frac{1}{\la 2\pi (\eta-\Phi_y(y,\xi),x-\Phi_\xi(y,\xi)) \ra^{2N}}\int_{\rdd}e^{2\pi i(\eta-\Phi_y(y,\xi),x-\Phi_\xi(y,\xi))\cdot(r,t)}\\
		&\qquad\times (1-\Delta_u)^N\tilde\tau(y,\xi,r,t)drdt.
	\end{align*}
	The same estimates of Theorem \ref{T4.1} yield to:
	\begin{align*}
		|k_{II}(x,\xi,y,\eta)|&\leq\frac{1}{\la 2\pi (\eta-\Phi_y(y,\xi),x-\Phi_\xi(y,\xi)) \ra^{2N}}\int_{\rdd}|(1-\Delta_u)^N\tilde\tau(y,\xi,r,t)|drdt\\
		&\asymp \frac{\la (y,\xi) \ra^{2N+m}}{\la(y,\eta)-\chi(x,\xi))  \ra^{2N}}.
	\end{align*}
\end{proof}

From \cite{CRPartI2022} we deduce
\begin{corollary}\label{cor42II}
	Under the assumptions of Theorem \ref{stimeII},   the estimate \eqref{nucleoFIO}
	holds true, hence $T_{II}\in FIO(\chi,N)$.
\end{corollary}
\begin{proof}
It follows from \eqref{50}, since $2N+m<0$ so that $\la (y,\xi)\ra^{2N+m}\leq 1$, for every $y,\xi\in\rd$.
\end{proof}

\section*{Acknowledgements}
The first three authors have been supported by the Gruppo Nazionale per l’Analisi Matematica, la Probabilità e le loro Applicazioni (GNAMPA) of the Istituto Nazionale di Alta Matematica (INdAM).

\end{document}